\documentclass[a4paper,12pt]{article}
\usepackage{a4}
\usepackage{amsmath}
\usepackage{amsfonts}
\usepackage{amsthm}
\usepackage{amssymb}
\usepackage[usenames,dvipsnames,svgnames,table]{xcolor}
\usepackage{graphicx}
\usepackage{tikz}
\usepackage{tkz-graph}
\usepackage{tkz-berge}
\usetikzlibrary{arrows,shapes}
\usepackage{url}
\usepackage{pgf}
\usepackage{tikz}
\usepackage{algorithm}
\usepackage{algorithmic}

\usepackage{caption}
\usepackage{subcaption} 
\usepackage{float}  

\usepackage{comment}

\usepackage{amsmath,stackrel}

\makeatletter
\newtheorem{theorem}{Theorem}[section]

\newtheorem{lemma}[theorem]{Lemma}
\newtheorem{prop}[theorem]{Proposition}

\newtheorem{cor}[theorem]{Corollary}

\newtheorem{example}[theorem]{Example}
\newtheorem{prm}[theorem]{Problem}
\newtheorem{rmk}[theorem]{Remark}
\newtheorem{conj}[theorem]{Conjecture}
\newtheorem*{pf}{Proof}

\def\ni{\noindent}

\def\thm{\textbftheorem}
\def \bp {\begin{prp} \ }
\def \ep {\end{prp}}
\def \bpm {\begin{prm} \ }
\def \epm {\end{prm}}
\def \bc {\begin{crl} \ }
\def \ec {\end{crl}}
\def \thm {\begin{Theorem} \ }
\def \ethm {\end{Theorem}}
\def \bl {\begin{lem} \ }
\def \el {\end{lem}}
\def \bd {\begin{defi} \ \rm }
\def \ed {\end{defi}}
\def \brm {\begin{rmk} \ }
\def \erm {\end{rmk}}
\def \bxm {\begin{xmp} \ \rm }
\def \exm {\end{xmp}}
\def \bcj {\begin{conj}}
\def \ecj {\end{conj}}
\def \nmr {\begin{enumerate}}
\def \enmr {\end{enumerate}}
\def \tmz {\begin{itemize}}
\def \etmz {\end{itemize}}

\begin{document}
\begin{center}
\huge{\textbf{Injective Edge Chromatic Index of a Graph}}
\end{center}

\begin{center}
Domingos M. Cardoso\footnote{dcardoso@ua.pt}
J. Orestes Cerdeira\footnote{jo.cerdeira@fct.unl.pt}
J. Pedro Cruz\footnote{pedrocruz@ua.pt} and
Charles Dominic\footnote{dominic@ua.pt}\\
$^{1,3,4}$Center for Research and Development in Mathematics and Applications\\
$^{2}$Department of Mathematics and Center of Mathematics and Applications (CMA), Faculty of Sciences and Technology,
New University of Lisbon, Quinta da Torre, 2829-516 Caparica, Portugal\\
$^{1,3}$Department of Mathematics, University of Aveiro, 3810-193 Aveiro, Portugal
\end{center}

\begin{abstract}
Three edges $e_{1}, e_{2}$ and $e_{3}$ in a graph $G$ are consecutive if they form a path (in
this order) or a cycle of length three. An injective edge coloring of a graph $G = (V,E)$
is a coloring $c$ of the edges of $G$ such that if $e_{1}, e_{2}$ and $e_{3}$ are consecutive edges in
$G$, then $c(e_{1})\neq c(e_3)$. The injective edge coloring number $\chi_{i}^{'}(G)$ is the minimum
number of colors permitted in such a coloring. In this paper, exact values of $\chi_{i}^{'}(G)$ for
several classes of graphs are obtained, upper and lower bounds for $\chi_{i}^{'}(G)$ are introduced
and it is proven that checking whether $\chi_{i}^{'}(G)= k$ is NP-complete.
\end{abstract}
\textbf{AMS Subject Classification 05C15}
\section{Introduction}
Throughout this paper we deal with simple graphs $G$ of \textit{order} $n\geq 2$ (the number of vertices) and
\textit{size} $m\geq 1$ (the number of edges). The vertex set and edge set will be denoted by $V(G)$ and $E(G)$, respectively.
A \textit{proper vertex (edge) coloring} of a graph $G$ is an assignment of colors to the vertices (edges) of
$G$, that is, $c:V(G) (E(G)) \rightarrow \mathcal{C}$,
where $\mathcal{C}$ is a set of colors, such that no two adjacent vertices (edges) have the same color, that is $c(x) \ne c(y)$ for every
edge $xy$ of $G$ ($c(e) \ne c(e')$ for every pair of edges $e, e'$ incident on the same vertex). The \textit{(edge) chromatic number}
($\chi'(G)$) $\chi(G)$ of $G$ is the minimum number of colors permitted in a such coloring.

Some variants of vertex and edge coloring have been considered.

An \textit{injective vertex coloring} of $G$ is a coloring of the vertices of $G$ so that any two
vertices with a common neighbor receive distinct colors.
The \textit{injective chromatic number} $\chi_{i}(G)$ of a graph $G$ is the smallest 
number of colors in 
an injective coloring of $G$.
Injective coloring of graphs was introduced by Hahn et. al in \cite{inj1} and was originated from Complexity Theory on Random Access Machines,
and can be applied in the theory of error correcting codes \cite{inj1}. In \cite{inj1} it was proved that, for $k \geq 3$, it is  NP-complete
to decide whether the injective chromatic number of a graph is at most $k$. Note that an injective coloring is not necessarily a proper coloring,
and vice versa (see {\cite{inj1,inj2}}).

The following variant of edge coloring was proposed in \cite{escd1}.
In a graph $G$, three edges $e_{1}$, $e_{2}$ and $e_{3}$ (in this fixed order) are called \textit{consecutive} if $e_{1}=xy$, $e_{2}=yz$ and
$e_{3}=zu$ for some vertices $x,y,z,u$ (where $x=u$ is allowed). In other words, three edges are consecutive if they form a path or a
cycle of lengths 3. A 3-consecutive edge coloring is a coloring of the edges such that for each three consecutive edges, $e_{1}$, $e_{2}$ and $e_{3}$,
the color of $e_2$ is one of the colors of $e_1$ or $e_3$. The 3-\textit{consecutive edge coloring number} of a graph $G$,
$\psi_{3c}^{'} (G)$, is the maximum number of colors of a 3-consecutive edge coloring of $G$. This concept was introduced and studied in
some detail in {\cite{escd1}}, where it is proven that the determination of the 3-consecutive edge coloring number for arbitrary graphs is
NP-hard.

Now we introduce the concept of \textit{injective edge coloring} (i-edge coloring for short) of a graph $G$  as a coloring,
$c:E(G) \rightarrow \mathcal{C}$,
such that if $e_{1}, e_{2}$ and  $e_{3}$ are consecutive edges in $G$, then $c(e_{1}) \ne c(e_3)$. The \textit{injective edge coloring number}
or \textit{injective edge chromatic index} of graph $G$, $\chi_{i}^{'}(G)$, is the minimum number of colors permitted in an i-egde coloring.
We say that graph $G$ is \textit{$k$ edge i-colorable} if  $\chi_{i}^{'}(G) \le k$.
Note that an i-egde coloring is not necessarily a proper edge coloring, and vice versa.

The motivation for the i-edge coloring is the following. We can model a Packet Radio Network (PRN) as an undirected graph $G=(V,E)$, where
the vertices represent the set of stations and two vertices are joined by an edge if and only if the corresponding
stations can hear each other transmissions, i.e, the set of edges $E$ represents the common channel property between the pairs of stations
(see \cite{Namdogopal_et_al2000, Ramanathan_Lloyd93}). Assigning channels or frequencies to the edges of $G$ we may define the secondary
interference as the one obtained when two
stations $x$ and $y$ that hear each other share the same frequency with one neighbor $x' \ne y$ of $x$ and one neighbor $y' \ne x$ of $y$.
An assignment of channels or frequencies to the edges between stations to avoid secondary interference corresponds to the i-edge coloring
of the graph (where each color is a frequency or channel).

In this paper we start the study of $i$-edge coloring of a graph $G$. We obtain exact values of $\chi_{i}^{'}(G)$ for several
classes of graphs, give upper and lower bounds for $\chi_{i}^{'}(G)$, and we prove that checking whether $\chi_{i}^{'}(G)= k$ is NP-complete.

For basic graph terminology we refer the book \cite{fh}. 

\section{Exact values of $\chi_{i}^{'}(G)$ for some classes of graphs}

We start this section with a few basic results 
which are direct consequences of the definition of injective edge coloring number.
As usually, the path, the cycle and the wheel with $n$ vertices will be denoted
by $P_n,$ $C_n$ and $W_n$, respectively. (The wheel with $n$ vertices is obtained by
connecting a single vertex to all vertices of $C_{n-1}$.) The complete graph of order $n$ is denoted $K_n$ and the complete
bipartite graph with bipartite classes of sizes $p$ and $q$
is denoted $K_{p, q}$. When $p=1$, the complete bipartite graph $K_{1,q}$ is called the star of order
$q+1$. (In particular, $P_2$ is the star $K_{1,1}$ and $P_3$ is the star $K_{1,2}$.)

Considering the above notation and denoting the Petersen graph by $\mathcal{P}$, the following
values for the injective edge coloring number can be easily derived.
\begin{prop}\label{particular_cases}
~
\\
\begin{enumerate}
\item  $\chi_{i}^{'}(P_{n})= 2$, for $n \ge 4$.\label{case_1}
\item $\chi_{i}^{'}(C_{n})= \left\{\begin{array}{ll}
                                          2, & \hbox{if } n \equiv 0 ~(\mbox{mod}~ 4),\\
                                          3, & \hbox{otherwise.}
                                   \end{array}\right.$ \label{case_2}
\item $\chi_{i}^{'}(K_{p, q})=\min \{p, q\}$.
\item 
$\chi_{i}^{'}(\mathcal{P})$ = $5$. (A feasible 5 i-egde coloring of the Petersen graph is shown in Figure~\ref{petersen_graph}. Note that
no pair of the edges labeled 1 to 5 can receive the same color.)
\end{enumerate}
\end{prop}

\begin{center}
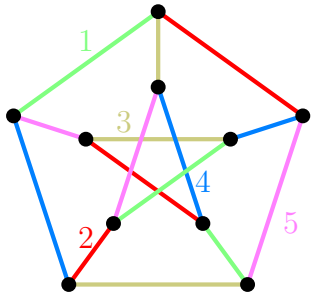
\begin{figure}[th]
\centering
\definecolor{c1}{rgb}{1.00,0.00,0.00}
\definecolor{c2}{rgb}{0.50,1.00,0.50}
\definecolor{c3}{rgb}{0.00,0.50,1.00}
\definecolor{c4}{rgb}{0.8,0.8,0.50} 
\definecolor{c5}{rgb}{1.00,0.50,1.00}
\begin{tikzpicture}[style= ultra thick]
\draw[c1] (18:2cm) -- (90:2cm);
\draw[c2] (90:2cm) -- (162:2cm) node[above,midway,c2] {1};
\draw[c3] (162:2cm)--(234:2cm);
\draw[c4] (234:2cm)--(306:2cm);
\draw[c5] (306:2cm)--(18:2cm) node[below,midway,c5] {$\quad5$};
\draw[c4] (18:1cm) -- (162:1cm); \node[] at (-.45, .54) {\textcolor[rgb]{0.8,0.8,0.50}{3}};
\draw[c1] (162:1cm) -- (306:1cm);
\draw[c3] (90:1cm) -- (306:1cm) ; \node[] at (.59,-.25) {\textcolor[rgb]{0.00,0.50,1.00}  {4}};
\draw[c5] (90:1cm) --(234:1cm);
\draw[c2] (234:1cm) --(18:1cm); %
\draw[c3] (18:1cm)--(18:2cm);
\draw[c4] (90:1cm)--(90:2cm);
\draw[c5] (162:1cm)--(162:2cm);
\draw[c1] (234:1cm)--(234:2cm); \node[] at (-.95,-1){\textcolor[rgb]{1.00,0.00,0.00} {2}};
\draw[c2] (306:1cm)--(306:2cm);
\foreach \x in {18,90,162,234,306}{
\draw (\x:2cm) [fill] circle  [radius=0.07];
\draw (\x:1cm) [fill] circle  [radius=0.07];}
\end{tikzpicture}
\caption{An injective edge coloring of Petersen graph with five colors. 
}\label{petersen_graph}
\end{figure}
\end{center}

\begin{prop}\label{coloring_edge_partition}
Let $G$ be a graph. Then $\chi_{i}^{'}(G)=k$ if and only if $k$ is the minimum positive integer for which
the edge set of $G$, $E(G)$, can be partitioned into non-empty subsets $E_1, \ldots, E_k$, such that the end-vertices
of the edges of each of these subsets $E_j$ induces a subgraph $G_j$ of $G$ where each component is a star.
\end{prop}

\begin{proof}
Let as assume that $\chi_{i}^{'}(G)=k$ and consider an injective edge coloring of the edges of $G$ using $k$
colors, $c_1, \ldots, c_k$. Then $E(G)$ can be partitioned into the subset of edges $E_1, \ldots, E_k$, where
every edge in $E_j$ has the color $c_j$, for $j=1, \ldots, k$. Then, for each $j \in \{1, \ldots, k\}$ the
end vertices of the edges of $E_j$ must induce a graph without three consecutive edges (otherwise the color
can not be the same for all the edges in $E_j$). Therefore, each component of the graph $G_j$ induced by the
end vertices of the edges in $E_j$ are stars. Furthermore, the positive integer $k$ is minimum (otherwise if
there exists such partition of $E(G)$ into $k'<k$ subsets of edges $E'_1, \ldots, E'_k$ then
$\chi_{i}^{'}(G)\le k' < k$).\\
Conversely, let us assume that $k$ is the minimum positive integer for which $E(G)$ can be partitioned as
described. Then, taking into account the first part of this proof, it is immediate that $\chi_{i}^{'}(G)=k$.
\end{proof}

Applying the Proposition~\ref{coloring_edge_partition}, we may conclude that the injective edge coloring
number of a wheel $W_{n},$, with $n \ge 4$ vertices, is:
$$
\chi_{i}^{'}(W_{n}) = \begin{cases}
                       6 \hskip.2cm if \hskip.2cm n \hskip.2cm is \hskip.2cm even \\
                       4 \hskip.2cm if \hskip.2cm n \hskip.2cm is \hskip.2cm odd \hskip.2cm and \hskip.2cm n-1 \equiv 0~(\mbox{mod}~ 4)\\
                       5 \hskip.2cm if \hskip.2cm n \hskip.2cm is \hskip.2cm odd \hskip.2cm and \hskip.2cm n-1 \not \equiv 0~(\mbox{mod}~ 4)
                      \end{cases}
$$
From Proposition~\ref{coloring_edge_partition}, it follows that whenever 
$\chi_{i}^{'}(G)=k$, the adjacency matrix $A_G$ of graph $G$ can be given by
\begin{equation}
A_G = \sum_{j=1}^{k}{A_{G_j}},\label{adjacency_matrix}
\end{equation}
where each $G_j$ is an induced subgraph of $G$, with at least one edge, and its components are stars or isolated vertices.
Therefore, $\chi_{i}^{'}(G_j)=1$, for $j=1, \ldots, k$, and $\chi_{i}^{'}(G)$ is the minimum number of induced subgraphs $G_j$
satisfying the conditions of Proposition~\ref{coloring_edge_partition}.

It is straightforward to see that for the edge chromatic number of $G$ and the vertex chromatic number of its line
graph $L(G)$, the equality $\chi^{'}(G)=\chi(L(G))$ holds. However, 
it is not always true that $\chi_{i}^{'}(G)=\chi_{i}(L(G))$. For instance, 
$\chi_{i}^{'}(K_{1,n})=1$ and $\chi_{i}(L(K_{1,n}))=n$.\\

Now, let us characterize the extremal graphs with largest and smallest injective chromatic index.

\begin{prop}\label{prop_2}
For any graph $G$ of order $n \ge 2$, $ \chi_{i}^{'}(G)=1$ if and only if $G$ is the disjoint union of $k\geq 1$ stars, i.e.,
$G=\cup_{j=1}^{k} K_{1, l_j}$, with $\sum_{j=1}^{k} l_j=n-k$ and $V(K_{1, l_j})\cap V(K_{1, l_{j'}})=\emptyset,$ for $j\not=j'$.
\end{prop}

\begin{proof}
Suppose $G$ is connected. If $G$ is $K_{1,n-1}$, then, by Proposition~\ref{particular_cases} - item 3, $\chi_{i}^{'}(G)=1$.
The converse follows taking into account that $\chi_{i}^{'}(G)=1$ implies that there are no three
consecutive edges in $G$ (otherwise $\chi_{i}^{'}(G) \ge 2$) and, since $G$ is connected, it must be $K_{1,n-1}$.\\
If $G$ is not connected, apply the proof above to each connected component, and take into account that edges
of different components can receive the same color.
\end{proof}

A trivial upper bound on the injective edge chromatic number of a graph $G$ is its size, that is, $\chi_{i}^{'}(G) \leq  |E(G)|$.
The Proposition~\ref{complete_graphs} characterizes the graphs for which this upper bound is attained.

\begin{prop}\label{complete_graphs}
Consider a graph $G$ of order $n$ and size $m$, with no isolated vertices.
Then $\chi_{i}^{'}(G)=m$ if and only if $G$ is complete.
\end{prop}
\begin{proof}
Assume that $G$ is 
the complete graph $K_n$, and consider two arbitrary edges $e_i$ and $e_j$
of $K_n$. Then either $e_i$ is adjacent to $e_j$ and thus they are both included in a triangle or there exists an edge $e_k$ such
that $e_i,$ $e_k$ and $e_j$ are three consecutive edges. In any of these cases $e_{i}$ and $e_{j}$ must have different colors.
Therefore, we have $\chi_{i}^{'}(G)=n(n-1)/2=m$.

Conversely, let us assume that $\chi_{i}^{'}(G)=m$. Clearly $G$ has to be connected, since otherwise
the same color could be used on edges from different components.
If $G$ has size one then it is complete.
Let us suppose that the size of $G$ is greater than one and $G$ is not complete. Since $G$ is connected,
there are two adjacent edges in $G$ not lying in the same triangle. Coloring these two edges by the color $c_1$ and all the
remaining edges differently, we produce an injective edge coloring with less than $m$ colors,
which is a contradiction. Therefore $G$ is complete.
\end{proof}

\section{${\boldmath{\omega^{'}}}$ edge injective colorable graphs}\label{sec_3}
The clique number of a graph $G$, denoted by $\omega(G)$, is the number of vertices in a maximum clique of $G$.
The number of edges in a maximum clique of $G$ is denoted by $\omega^{'}(G)$.
If $G$ has size $m\geq 1$, $\omega^{'}(G)=\frac{\omega(G)(\omega(G)-1)}{2}$.
\begin{prop}\label{PropEIP1}
For any connected graph $G$ of order $n \ge 2$, $\chi_{i}^{'}(G)\geq \omega^{'}(G)$.
\end{prop}
\begin{proof}
Let $K_{r}$ be a maximum clique in $G$. From Proposition~\ref{complete_graphs},
$\chi_{i}^{'}(K_{r})= r(r-1)/2= \omega^{'}(G)$.
Therefore, we need at least $r(r-1)/2$ colors to color the edges of $G$, i.e.,
$\chi_{i}^{'}(G)\geq r(r-1)/2=\omega^{'}(G)$.
\end{proof}

We say that $G$ is an \textit{$\omega^{'}$ edge injective colorable} ($\omega^{'}$EIC-)graph if
$\chi_{i}^{'}(G)=\omega^{'}(G)$.

\begin{example}
The following graphs are examples of $\omega^{'}$EIC-graphs.
\begin{enumerate}
\item The complete graph, $K_n$.
\item The star, $K_{1,q}$.
\item The friendship graph, i.e., the graph with $n=2p+1$ vertices formed by $p\geq 1$ triangles all attached to a common vertex
(see Figure~\ref{friendship_graph}).
\end{enumerate}
\end{example}
\begin{center}
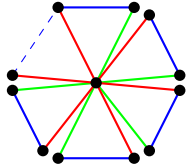
\begin{figure}
\centering
\begin{tikzpicture}
[every node/.append style={circle,  fill, inner sep=1.5pt, minimum size=1.5pt}]
\draw [thick,color=red,step=.5cm,](4,10)--(3.5,11);
\draw [thick,color=green,step=.5cm,](4,10)--(4.5,11){};
\draw [thick,color=red,step=.5cm,](4,10)--(4.7,10.9){};
\draw [thick,color=green,step=.5cm,](4,10)--(5.1,10.1) {};
\draw [thick,color=red,step=.5cm,](4,10)--(5.1,9.9) {};
\draw [thick,color=green,step=.5cm,](4,10)--(4.7,9.1) {};
\draw [thick,color=red,step=.5cm,](4,10)--(4.5,9)  {};
\draw [thick,color=green,step=.5cm,](4,10)--(3.5,9)  {};
\draw [thick,color=red,step=.5cm,](4,10)--(3.3,9.1)  {};
\draw [thick,color=green,step=.5cm,](4,10)--(2.9,9.9) {};
\draw [thick,color=red,step=.5cm,](4,10)--(2.9,10.1) {};
\draw [thick,color=blue,step=.5cm,] (4.5,11)--(3.5,11);
\draw [thick,color=blue,step=.5cm,](4.7,10.9)--(5.1,10.1) {};
\draw [thick,color=blue,step=.5cm,](5.1,9.9)--(4.7,9.1);
\draw [thick,color=blue,step=.5cm,](4.5,9)--(3.5,9)  {};
\draw [thick,color=blue,step=.5cm,](3.3,9.1)--(2.9,9.9) {};
\draw [dashed,color=blue,step=.5cm,](2.9,10.1)--(3.5,11)  {};
\node (1) at (4,10)  {};
 \node (2) at (3.5,11)  {};
 \node (3) at (4.5,11){};
 \node (4) at (4.7,10.9){};
 \node (5) at (5.1,10.1) {};
 \node (6) at (5.1,9.9) {};
 \node (7) at (4.7,9.1)  {};
 \node (8) at (4.5,9)  {};
 \node (9) at (3.5,9)  {};
 \node (10) at (3.3,9.1)  {};
 \node (11) at (2.9,9.9) {};
 \node (12) at (2.9,10.1) {};
\end{tikzpicture}
\caption{The friendship graph.}\label{friendship_graph}
\end{figure}\end{center}


\begin{prop}\label{coalescing}
For any positive integer $p \geq 3$, consider the complete graph $K_{p}$, with  $V(K_p)=\{v_1, \ldots, v_p\}$,
and a family of stars $K_{1,q_1}, \ldots, K_{1,q_p}$, with $q_j\geq 1$. Let $G$ be the graph obtained coalescing a maximum degree
vertex of the star $K_{1,q_j}$ with the vertex $v_j$ of $K_{p}$, for $j=1, \ldots, p$. Then $G$ is an $\omega^{'}$EIC-graph.
\end{prop}

\begin{proof}
Consider the Hamiltonian cycle of $K_p$, $C_p=v_1,v_2 \ldots v_p, v_1$. Color each edge $e_i=v_iv_{i+1},$
for $i=1, \ldots, p-1$ of  $C_p$ with color $c_i$, $1 \le i \le p-1$ and the edge $e_p=v_pv_1$ with
color $c_p$. For each $j \in \{1, \ldots, p\}$, color all the edges of the star $K_{1,q_j}$ with color
$c_j$. Now color all the remaining edges of $K_p$ differently. Since this coloring produces an injective
edge coloring, we have
\begin{equation}
\chi_{i}^{'}(G) \le p(p-1)/2=\omega^{'}(K_{p})=\omega^{'}(G). \label{Kp-coalescing_2}
\end{equation}
The result now follows from Proposition \ref{PropEIP1}
\end{proof}

Notice that the \textit{corona} $K_p \circ K_1$, that is, the graph obtained from $K_p$ by adding a pendant
edge to each of its vertices, is a particular case of the graphs $G$ considered in Proposition~\ref{coalescing},
which is obtained setting $K_{1,q_j}=K_{1,1}$ for $j=1, \ldots, p$.

\begin{prop}\label{unicyclic_eip_grapj}
If $G$ is a unicyclic graph with $K_{3}$, then $G$ is an $\omega^{'}$EIC-graph.
\end{prop}
\begin{proof} Let the vertices of the cycle $K_{3}$ be $v_{1},v_{2},v_{3}$, and the edges
$e_{1}=v_{1}v_{2}$, $e_{2}=v_{2}v_{3}$, $e_3=v_3v_1$. Color the edge $e_i$ with color $c_i$, for $i=1,2,3$.
Let $T_{1}$, $T_{2}$ and $T_{3}$ be the trees which are incident to $v_{1}$, $v_{2}$ and $v_{3}$, respectively, and color the edges of these trees as follows.
\begin{itemize}
\item{
} 
Color the edges in $T_{1}$ which are incident to $v_{1}$ with color $c_1$, and call $C^1_{1}$ the set of these edges.
Consider the edges in $T_{1}$ which are adjacent to $C^1_{1}$ edges, and color all these edges with color $c_2$.
Call $C^1_{2}$ the set of these edges.
Now consider the edges in $T_{1}$ which are adjacent to edges of $C^1_{2}\setminus C^1_{1}$, color these edges with color $c_3$,
and call the set of these edges $C^1_{3}$.
Again, consider the edges in $T_{1}$ which are adjacent to edges of $C^1_{3}\setminus C^1_2$, color these edges with color $c_1$, and call
the set of these edges $C^2_{1}$ edges. Continue this procedure
until all edges in $T_{1}$ have been colored.
\item{
}
Color the edges in $T_{2}$ which are incident to $v_{2}$ with color $c_2$, and call the set of these edges $C^1_{2}$.
Consider the edges in $T_{2}$ which are adjacent to $C^1_{2}$ edges, color all these edges with color $c_3$, and call $C^1_3$ the
set of  these edges.
Now consider the edges in $T_{2}$ which are adjacent to $C^1_3\setminus C^1_{2}$, color these edges with color $c_1$,
and call $C^1_1$ the set of these edges.
Again, consider the edges in $T_{2}$ which are adjacent to $C^1_1\setminus C^1_{3}$ edges, color these edges
    with color $c_2$, and call the set of these edges $C^2_2$.
Continue this procedure
until all edges in $T_{2}$ have been colored.
\item{
}
Color the edges in $T_{3}$ which are incident to $v_3$ with color $c_3$ and call the set of
these    edges $C^1_3$.
Consider the edges in $T_{3}$ which are adjacent to $C^1_{3}$ edges, color these edges with color $c_1$, and call the
set of these edges $C^1_{1}$.
Now consider the edges in $T_{3}$ which are adjacent to $C^1_{1}\setminus C^1_{3}$ edges, color these edges with color $c_2$, and call  $C^1 _{2}$
the set of these edges.
Again, consider the edges in $T_{3}$ which are adjacent to $C^1_{2}\setminus C^1_{1}$, color these edges
    with color $c_3$ and call the set of these edges $C^2_{3}$.
Continue this procedure
until all edges in $T_{3}$ have been colored.
\end{itemize}
This clearly produces a feasible 3 i-egde coloring of $G$, and since 3 colors are needed for coloring
the triangle $K_3$, we can conclude that $\chi_{i}^{'}(G)= 3= \omega^{'}(G)$, and the result follows.
\end{proof}

\section{Bounds on the injective chromatic index}\label{sec_4}


%
%
%
%

Now we consider the injective edge coloring number of bipartite graphs.

\begin{prop}\label{bound_on_bipertite_graphs}
If $G$ is a bipartite graph with bipartition $V(G)=V_{1}\cup V_{2}$, and $G$ has no isolated vertices,
then $\chi_{i}^{'}(G)\le \min \{\mid V_{1}\mid, \mid V_{2}\mid\}$.
\end{prop}

\begin{proof}
The proof follows directly from Proposition \ref{particular_cases} - item 3.
\end{proof}
Note that the above bound is attained for every complete bipartite graph
$K_{p,q}$.




We now combine Proposition \ref{bound_on_bipertite_graphs} with results from \cite{escd1} on the $3$-consecutive edge
coloring of graphs to obtain bounds on the injective edge chromatic index for bipartite graphs.

Bujt\'as et. al  \cite{escd1} proved the following results.

\begin{prop}\label{Buj1}
\cite{escd1}
If $G$ is a bipartite graph with bipartition $V(G) = V_1 \cup V_2$, and $G$ has no isolated vertices,
then $\max\{|V_{1}|, |V_{2}|\} \le \psi_{3c}^{'} (G) \le \alpha(G)$, where $\alpha(G)$ is the
independence number of $G$.
\end{prop}

\begin{prop}\label{Buj2} \cite{escd1}
Let $G$ be a graph of order $n$.
\begin{itemize}
\item If $G$ is connected, then $\psi_{3c}^{'} (G )\le  n - \frac{n-1}{\Delta(G)}$ , where $\Delta(G)$ denotes the maximum degree of $G$;
\item $\psi_{3c}^{'} (G ) \le  n- i(G)$, where $i(G)$ is the \textit{independence domination number} of $G$, i.e., the minimum cardinality
      among all maximal independent sets of $G$.
\end{itemize}
\end{prop}

From Propositions \ref{bound_on_bipertite_graphs}, \ref{Buj1} and \ref{Buj2} we can directly conclude
the following.
\begin{cor}
Let $G$ be a connected bipartite graph of order $n \ge 2$. Then
\begin{eqnarray*}
\chi_{i}^{'}(G) & \le & n - \frac{n-1}{\Delta(G)},\\
\chi_{i}^{'}(G) & \le & n- i(G),\\
\chi_{i}^{'}(G) & \le & \alpha(G).
\end{eqnarray*}
\end{cor}

Now we introduce an upper bound on the edge injective coloring number of a graph $G$ in terms of its size
and \textit{diameter}, which we denote by $\text{diam}(G)$.

\begin{prop}
For any connected graph $G$ of size $m\geq 3$, $\chi_{i}^{'}(G) \le m-\text{diam}(G)+2$. This upper bound is
attained if and only if $G$ is 
$P_{m+1}$.
\end{prop}
\begin{proof}
Let $P_d$ be a diametral path of $G$. We can color the path $P_d$ with 
2 colors.
Coloring all the other edges differently with $m-\text{diam}(G)$ colors, we produce an injective edge coloring of $G$,
and then $\chi_{i}^{'}(G)\leq m-\text{diam}(G)+2$.\\
The proof of the last part of this proposition can be divided in two cases:
\begin{enumerate}
\item 
If $G$ is a path, with $n \ge 4$, it follows from Proposition~\ref{particular_cases}- item 1 that $\chi_{i}^{'}(P_n) = 2$ and, since $\text{diam}(G)=m$,
the result holds.
\item Let us assume that $G$ is not a path.
      \begin{itemize}
      \item If $\text{diam}(G)\leq 2$, Proposition \ref{complete_graphs}  implies that 
            $$
            \chi_{i}^{'}(G) < m-\text{diam}(G)+2.  
            $$
            In fact, if $\text{diam}(G)=1$, then $G$ is complete and thus
            $m-\text{diam}(G)+2 > m = \chi_{i}^{'}(G)$. If $\text{diam}(G)=2$, then $G$ is not complete and
            thus $m-\text{diam}(G)+2 = m > \chi_{i}^{'}(G)$.
      \item If $\text{diam}(G)> 2$, 
            consider a diametral path $P_d=x_1, \ldots, x_{d+1}$.
            Since $G$ is connected and is not a path, then there exists a vertex $u \not \in V(P_d)$ which has (i) one, (ii) two
            or at most (iii) three neighbors in $P_d$, otherwise $P_d$ is not diametral.
            \begin{enumerate}
            \item[(i)] Suppose $u$ has a unique neighbor, say $x_i$, in $P_d$. As $P_d$ is a diametral path, $x_i$ has
                   to be an interior vertex of $P_d$, i.e., $i\not=1, d+1$, and
                the edges of $P_d$ can be colored in a way
                  such that $x_{i-1}x_i$ and $x_ix_{i+1}$ have the same color $c$ and this color $c$ can also be
                  used for coloring the edge $ux_i$.
                  The remaining $m-\text{diam(G)}-1$ edges can be colored with no more than $m-\text{diam(G)}-1$ colors, and
                  thus producing an i-egde coloring with at most $2 + m-\text{diam}(G)-1$ colors, and therefore
                  $\chi_{i}^{'}(G) < m-\text{diam}(G)+2$.
            \item[(ii)] If $u$ has two neighbors in $P_d$, say $x_i$ and $x_j$, then they must have at most one vertex
                  between them, i.e., $j=i+1$ or $j=i+2$. If $j=i+2$, the two edges $ux_i$ and $ux_{j}$ can be colored
                  with the same color, different from each of the two colors used for the edges of $P_d$, and using a different
                  color for each of the $m-\text{diam(G)}-2$ other edges. Thus,
                 $\chi_{i}^{'}(G) \leq 2 +1 +m-\text{diam(G)}-2 < m-\text{diam}(G)+2$.
If $j=i+1$, then use two colors to color the edges on path $(P_d\setminus x_ix_j)\cup ux_i\cup ux_j$, a new color for edge $x_ix_j$
and a different color for each of the remaining $m -\text{diam(G)}-2$ edges. Again,
$\chi_{i}^{'}(G) \leq 2 +1 +m-\text{diam(G)}-2 < m-\text{diam}(G)+2$.
            \item[(iii)] If $u$ has three neighbors in $P_d$, then they must be consecutive (otherwise $P_d$ is not
                  diametral), say $x_i, x_{i+1}, x_{i+2}$.  Coloring again the edges of $P_d$ 
                  using two
                  colors, say $c_1$ and $c_2$, edges $ux_i$ and $ux_{i+2}$ can be colored with an additional color $c_3$,
                  and edge $x_{i+1}u$ with a different color $c_4$.
                  Using a different color for each of the $m- \text{diam(G)}-3$, we can conclude that
                  $\chi_{i}^{'}(G) \le 2+1+1+ m-\text{diam}(G)-3 < m-\text{diam}(G) + 2$.
            \end{enumerate}
  \end{itemize}
\end{enumerate}
\end{proof}

\begin{prop}\label{ub_for_trees}
For any tree $T$ of order $n \ge 2$, $1 \leq \chi_{i}^{'}(T) \le 3$.
\end{prop}
\begin{proof}
If $n= 2$, $\chi_{i}^{'}(T) =1$.
If $n \geq 3$, an edge can be added to $T$ in such a way that the resulting graph $H$ includes a triangle. We then have,
$\chi_{i}^{'}(T) \le \chi_{i}^{'}(H) $, and using Proposition ~\ref{unicyclic_eip_grapj}, $ \chi_{i}^{'}(H) = 3$.
\end{proof}

These lower and upper bounds are sharp. According to Proposition~\ref{prop_2}, the stars are the only
connected graphs $G$ such that $\chi_{i}^{'}(G)=1$.
Regarding the upper bound, the tree $T'$ of  
Figure~\ref{T1_tree} is a minimum  size tree with $\chi_{i}^{'}(T')=3$.

\begin{figure}[th]
\centering
\begin{tikzpicture}[every node/.append style={circle,  fill, inner sep=1.5pt, minimum size=1.5pt},scale=1]
\draw (0,1) -- (1,1) -- (2,0.5) -- (3,0.5) -- (4,1) -- (5,1);
\draw (1,0) -- (2,0.5);
\draw (3,0.5) -- (4,0);
\node (pcan) at (0,1) {};
\node (pcan) at (1,0) {};
\node (pcan) at (1,1) {};
\node (pcan) at (2,0.5) {};
\node (pcan) at (3,0.5) {};
\node (pcan) at (4,0) {};
\node (pcan) at (4,1) {};
\node (pcan) at (5,1) {};
\end{tikzpicture}
\caption{a minimum  size tree $T'$ with $\chi_{i}^{'}(T')=3$.}\label{T1_tree}
\end{figure}
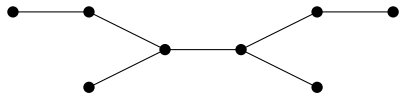

Since for a subgraph $H$ of a graph $G$, any 3-consecutive edges of $H$ are also 3-consecutive edges of $G$, we have the following result.
\begin{prop}\label{subgraph_result}
If $H$ is a  subgraph of a connected graph $G$, then $\chi_{i}^{'}(H)\leq\chi_{i}^{'}(G)$.
\end{prop}
As immediate consequence, we have the corollary below.
\begin{cor}\label{cor_10}
Let $G$ be a connected graph of order $n$. 
\begin{enumerate}
\item If $x$ is an edge of $G$, then $\chi_{i}^{'}(G) \le \chi_{i}^{'}(G+x).$\label{cor_10.1}
\item If $G$ includes a cycle $C_p$ and $4 \le p \not\equiv 0 ~ (\mbox{mod}~ 4)$, then $\chi_{i}^{'}{(G)}\ge 3$. \label{cor_10.3}
\item If $G$ includes a complete graph $K_{p}$, then $\chi_{i}^{'}{(G)} \ge p(p-1)/2$.\label{cor_10.4}
\item If $G$ includes the tree depicted in Figure~\ref{T1_tree}, then $\chi_{i}^{'}{(G)} \ge 3$.\label{cor_10.2}
\item If $G$ is a tree $T,$ which includes the subtree $T'$ depicted in Figure~\ref{T1_tree}, then $\chi_{i}^{'}(T)=3$.\label{cor_10.5}
\end{enumerate}
\end{cor}

In computer science a \textit{perfect binary tree} is a tree data structure with exactly one vertex of degree two and where each of the
other vertices has degree one or three. Now we have another corollary.
\begin{cor}
Let $T$ be a perfect binary tree with $diam(T) \ge 7$. Then $\chi_{i}^{'}(T)=3$.
\end{cor}
\begin{proof}
Note that every perfect binary tree $T$ with $diam(T) \ge 7$ has to include the tree depicted in Figure~\ref{T1_tree} as induced subgraph.
Therefore, from Corollary~\ref{cor_10} - item \ref{cor_10.5}, the result follows.
\end{proof}

We proceed deriving a characterization of the injective edge chromatic index, which is checkable
in polynomial time for trees.

Let $G$ be a graph of size $m\geq 1$, $\bar{G}$ the graph with $m$ vertices corresponding to the edges of $G$
and where, for every pair of vertices $x,y\in V(\bar{G})$, $xy\in E(\bar{G})$ if and only if there is an
edge $e\in E(G)$ such that $x, e, y$ are consecutive edges of $G$.  We obviously have the following.
\begin{lemma}\label{lemmaGbarG}
If $G$ is a graph of size $m\geq 1$, $\chi_{i}^{'}(G)= \chi(\bar{G})$.
\end{lemma}

From Lemma \ref{lemmaGbarG} we can conclude the following.
\begin{prop}\label{propGbarG}
If $G$ is a graph of size $m\geq 1$, then $\chi_{i}^{'}(G)\leq 2$ if and only if $\bar{G}$ is bipartite.
\end{prop}
\begin{proof}
Note that the chromatic number of a graph with no edges is 1, and is equal to 2 if and only if it has at least one edge and is
bipartite. Lemma \ref{lemmaGbarG} completes the proof.
\end{proof}

We therefore have the following characterization of graphs having injective edge chromatic index equal to 2.
\begin{prop}\label{charact_graphs}
Let $G$ be a graph of size $m\geq 1$. Then, $\chi_{i}^{'}(G)= 2$ if and only if  $G$ is not a disjoint union of stars
and $\bar{G}$ has no odd cycle.
\end{prop}
\begin{proof}
Proposition \ref{prop_2} states that $\chi_{i}^{'}(G)= 1$ if and only if $G$ is a disjoint union of stars.
Proposition  \ref{propGbarG} states that if $\bar{G}$ is bipartite then $\chi_{i}^{'}(G) \leq  2$.
\end{proof}

Tacking into account Propositions \ref{prop_2} and \ref{ub_for_trees},
Proposition \ref{charact_graphs} reads for trees as follows.

\begin{prop}\label{charact_trees}
Let $T$ be a tree. Then, either
\begin{itemize}
\item $\chi_{i}^{'}(T)= 1$ if $T$ is a star, or
\item $\chi_{i}^{'}(T)= 3$ if $\bar{T}$ includes an odd cycle, or
\item  $\chi_{i}^{'}(T)= 2$, in any other case.
\end{itemize}
\end{prop}

For example, the graph $\bar{T'}$ that is obtained from the tree $T'$ of
Figure~\ref{T1_tree}, which has $\chi_{i}^{'}(T')= 3$, includes cycles $C_5$ and $C_7$.

Note that Proposition \ref{charact_trees} gives a polynomial time algorithm
to determine the injective edge chromatic index for trees.

\vskip1.ex

The next result relates the injective edge chromatic index of a graph and of its square.

Let us denote the \textit{distance} between the vertices $u$ and $v$ in $G$ by $d_{G}(u, v)$.
The \textit{square} of a simple graph $G$ is the simple graph $G^{2}$, where $e=uv$ is an edge in $G^{2}$
if and only if $d_{G}(u, v)\leq 2$. Using this concept and this notation we have the corollary.

\begin{cor}
For any connected graph $G$, $\chi_{i}^{'}(G)\leq \chi_{i}^{'}(G^{2})$.
\end{cor}

\begin{proof}
Notice that $G$ is a subgraph of $G^2$. Therefore, applying Proposition~\ref{subgraph_result}, the result follows.
\end{proof}

Previously we have considered the unicyclic graphs which include a triangle and we proved
that those are  $\omega^{'}$EIC-graphs.
Now, the following proposition states a lower and upper bounds on the injective chromatic index
of more general unicyclic graphs.

\begin{prop}\label{unicyclic_injective_coloring}
Let $G$ be a unicyclic graph and $C_p$ the cycle in $G$. If $p \ge 4$, then  $2\leq \chi_{i}^{'}(G) \le 4$.
\end{prop}
\begin{proof}
The left inequality follows directly from Proposition \ref{particular_cases}, item 2.

To prove that $\chi_{i}^{'}(G) \le 4$, let $v$ be an arbitrary vertex of the cycle
$C_p$ and consider $G-v$ (the graph obtained from $G$ deleting $v$ and every edge of $G$
incident to $v$). As $G-v$ is a forest we can properly i-coloring its edges with
three colors, say colors $c_1, c_2,c_3$. Now use a different color, say $c_4$, to color
all edges of $G$ incident to vertex $v$. This is clearly a feasible i-egde coloring of
$G$ using 4 color, showing that the result holds.
\end{proof}

Notice that the upper bound on the injective edge coloring number obtained in
Proposition~\ref{unicyclic_injective_coloring} is attained for the unicyclic
graph $\aleph$ depicted in Figure~\ref{grid graph}. Regarding the lower bound,
it is attained for a graph $G$ if and only if $\bar{G}$ is bipartite.

\section{The injective chromatic index of some mesh graphs}\label{sec_5}

Herein we call mesh graphs the graphs considered in \cite{IVS}. Among these graphs we pay particular
attention to the cartesian products $P_{n} \Box K_2$ and $P_r \Box P_s$ and also to the honeycomb graph.
The \textit{Cartesian product} $G \Box H$ of two graphs $G$ and $H$ is the graph with vertex set equal to the
Cartesian product $V(G) \times V(H)$ and where two vertices $(g_1,h_1)$ and $(g_2,h_2)$ are adjacent in $G \Box H$
if and only if, either $g_1$ is adjacent to $g_2$ in $G$ or $h_1$ is adjacent to $h_2$ in $H$, that is, if $g_1=g_2$
and $h_1$ is adjacent to $h_2$ or $h_1=h_2$ and $g_1$ is adjacent to $g_2$.
\begin{prop}
Let $P_{n}$ be a path of order $n \geq 3 $.  Then $\chi_{i}^{'}(P_{n} \Box K_2) = 3$
\end{prop}

\begin{proof}
Mark all the vertices in $P_{n} \Box K_2$ from the left to right as follows: mark the first upper vertex in the
ladder by $1$, the second lower vertex by $2$, the third upper vertex by $3$ and so on, as it is in
Figure~\ref{cartesian_product}. Now color the edges with one end vertex labeled $1$ by the color $c_{1}$,
the edges with one end vertex labeled $2$ by the color $c_2$, the edges with one end vertex labeled $3$
by the color $c_3$ and so on. This coloring yields an injective edge coloring of $P_{n} \Box K_2$. Therefore,
$\chi_{i}^{'}(P_{n} \Box K_2) \le 3$. On the other hand, it is easy to find $C_6$ as a subgraph of $P_{n} \Box K_2$
and, from Proposition~\ref{particular_cases}-\ref{case_2}, $\chi_{i}^{'}(C_6)=4$. Then,  applying Proposition~\ref{subgraph_result},
it follows that $\chi_{i}^{'}(P_{n} \Box K_2) \ge 3$.
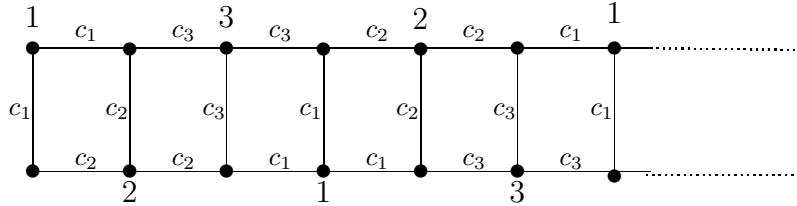
\begin{figure}[th]
\begin{center}
\unitlength .85mm 
\linethickness{0.2pt}
\ifx\plotpoint\undefined\newsavebox{\plotpoint}\fi 
\begin{picture}(100.25,30.5)(0,05)
\put(3.5,26.561){\line(1,0){95.75}}
\put(3.75,7.311){\line(1,0){95.25}}
\put(3.5,26.561){\line(0,-1){19.5}}
\put(18.5,26.561){\line(0,-1){19.5}}
\put(33.5,26.561){\line(0,-1){19.5}}
\put(48.5,26.561){\line(0,-1){18.75}}
\put(63.5,26.561){\line(0,-1){19}}
\put(78.5,26.561){\line(0,-1){19.25}}
\put(93.5,26.561){\line(0,-1){20}}
\put(3.5,7.311){\circle*{2.121}}
\put(3.5,26.561){\circle*{2.121}}
\put(18.5,26.31){\circle*{2.121}}
\put(18.5,7.311){\circle*{2.121}}
\put(33.5,26.561){\circle*{2.121}}
\put(33.5,7.311){\circle*{2.121}}
\put(48.5,26.311){\circle*{2.121}}
\put(48.5,7.311){\circle*{2.121}}
\put(63.5,26.31){\circle*{2.121}}
\put(63.5,7.311){\circle*{2.121}}
\put(78.5,26.561){\circle*{2.121}}
\put(78.5,7.311){\circle*{2.121}}
\put(93.5,26.561){\circle*{2.121}}
\put(93.5,6.561){\circle*{2.121}}

\multiput(98.43,6.68)(.97826,0){24}{{\rule{.4pt}{.4pt}}}
\multiput(98.18,26.43)(.95833,-.01042){25}{{\rule{.4pt}{.4pt}}}
\put(2.25,29.75){$1$}
\put(17.25,2.5){$2$}
\put(32.25,29.75){$3$}
\put(47.25,2.5){$1$}
\put(62.25,29.5){$2$}
\put(77.25,2.5){$3$}
\put(92.25,30.5){$1$}
\put(10,28.25){\footnotesize $c_{1}$}
\put(10,8.25){\footnotesize $c_{2}$}
\put(25,28.25){\footnotesize $c_{3}$}
\put(40,28.25){\footnotesize $c_{3}$}
\put(25,8.25){\footnotesize $c_{2}$}
\put(40,8.25){\footnotesize $c_{1}$}
\put(55,8.25){\footnotesize $c_{1}$}
\put(55,28.25){\footnotesize $c_{2}$}
\put(70,28.25){\footnotesize $c_{2}$}
\put(70,8.25){\footnotesize $c_{3}$}
\put(85,8.25){\footnotesize $c_{3}$}
\put(85,28.25){\footnotesize $c_{1}$}
\put(-0.25,16){\footnotesize $c_{1}$}
\put(14.75,16){\footnotesize $c_{2}$}
\put(29.75,16){\footnotesize $c_{3}$}
\put(44.75,16){\footnotesize $c_{1}$}
\put(59.75,16){\footnotesize $c_{2}$}
\put(74.75,16){\footnotesize $c_{3}$}
\put(89.75,16){\footnotesize $c_{1}$}
\end{picture}
\end{center}
\caption{The Cartesian product $P_{n} \Box K_{2}$ with $\chi_{i}^{'}{(G)}=3$.}\label{cartesian_product}
\end{figure}
\end{proof}

Before the next proposition it is worth to recall that a \textit{two dimensional grid} graph is the graph
obtained by the cartesian product $P_{r} \Box P_{s}$, where $r$ and $s$ are integers.

\begin{prop}
If $r, s \ge 4 $, then $\chi_{i}^{'}(P_{r} \Box P_{s})=4$.
\end{prop}
\begin{proof}
We start to choose a color, say red, for the edges having the upper left corner vertex $v$ of $G=P_{r}\Box P_{s}$
as end vertex. Then, we color the edges having as end vertex the vertices which form a diagonal
(starting at $v$) of the grid $G$ alternating between red and another color, say green. The parallel diagonals are
colored in the same way, using two different colors, say blue and yellow (see
Figure~\ref{grid graph}). It is easy to check that this coloring produces an injective edge coloring of $G$ and
therefore, $\chi_{i}^{'}(G) \le 4$. Since the graph $\aleph$ depicted in Figure~\ref{grid graph}
is a subgraph of $G$ such that $\chi_{i}^{'}(\aleph)=4$, applying Proposition~\ref{subgraph_result}, it follows
that $4 = \chi_{i}^{'}(\aleph) \le \chi_{i}^{'}(G) \le 4$.
\end{proof}
\begin{center}
\begin{figure}[th]

\centering
\begin{tikzpicture}[every node/.append style={circle,  fill, inner sep=1.5pt, minimum size=1.5pt},scale=1]
\draw [thick,color=blue,step=.5cm,] (1,10) -- (1,10.5);\draw [thick,color=red,step=.5cm,]  (1,10.5) -- (1,11);
\draw [thick,color=red,step=.5cm,]  (1,11) --(1,11.5);\draw [thick,color=blue,step=.5cm,]  (1,11.5) -- (1,12);
\draw [thick,color=blue,step=.5cm,]  (1,12) -- (1,12.5);\draw [thick,color=red,step=.5cm,]  (1,12.5) -- (1,13);
\draw [thick,color=blue,step=.5cm,]  (1,10) -- (1.5,10);\draw [thick,color=green,step=.5cm,]  (1,10.5) -- (1.5,10.5);
\draw [thick,color=green,step=.5cm,]  (1.5,10) -- (1.5,10.5);\draw [thick,color=green,step=.5cm,]  (1.5,10.5) -- (1.5,11);
\draw [thick,color=red,step=.5cm,]  (1,11) -- (1.5,11);\draw [thick,color=yellow,step=.5cm,]  (1.5,11.5) -- (1.5,11);
\draw [thick,color=yellow,step=.5cm,]  (1.5,11.5) -- (1,11.5); \draw [thick,color=yellow,step=.5cm,]  (1.5,11.5)-- (1.5,12);
\draw [thick,color=blue,step=.5cm,]  (1,12)-- (1.5,12) ;\draw [thick,color=green,step=.5cm,] (1.5,12)--(1.5,13);
\draw [thick,color=green,step=.5cm,] (1.5,12.5)-- (1,12.5);\draw [thick,color=red,step=.5cm,] (1,13)--(1.5,13);
\draw [thick,color=red,step=.5cm,] (1.5,10)-- (2,10);\draw [thick,color=red,step=.5cm,] (2,10)--(2,10.5);
\draw [thick,color=red,step=.5cm,] (2,10)-- (2.5,10);\draw [thick,color=green,step=.5cm,] (2,10.5)-- (1.5,10.5);
\draw [thick,color=blue,step=.5cm,](2,11)--(2,10.5);\draw [thick,color=blue,step=.5cm,](2,11)--(1.5,11);
\draw [thick,color=blue,step=.5cm,](2,11)--(2,11.5);\draw [thick,color=blue,step=.5cm,](2,11)--(2.5,11);
\draw [thick,color=yellow,step=.5cm,] (1.5,11.5)--(2,11.5);\draw [thick,color=red,step=.5cm,] (2,12)--(2,11.5);
\draw [thick,color=red,step=.5cm,] (2,12)--(1.5,12);\draw [thick,color=red,step=.5cm,] (2,12)--(2,12.5);
\draw [thick,color=red,step=.5cm,] (2,12)--(2.5,12); \draw [thick,color=green,step=.5cm,] (1.5,12.5)--(2,12.5);
\draw [thick,color=blue,step=.5cm,] (2,13)--(1.5,13); \draw [thick,color=blue,step=.5cm,] (2,13)-- (2,12.5);
\draw [thick,color=blue,step=.5cm,] (2,13)-- (2.5,13); \draw [thick,color=yellow,step=.5cm,](2.5,10.5)--(2.5,10);
\draw [thick,color=yellow,step=.5cm,](2.5,10.5)--(2,10.5);\draw [thick,color=yellow,step=.5cm,](2.5,10.5)--(2.5,11);
\draw [thick,color=yellow,step=.5cm,](2.5,10.5)--(3,10.5);\draw [thick,color=green,step=.5cm,](2.5,11.5)--(2.5,11);
\draw [thick,color=green,step=.5cm,](2.5,11.5)--(2,11.5);\draw [thick,color=green,step=.5cm,](2.5,11.5)--(2.5,12);
\draw [thick,color=green,step=.5cm,](2.5,11.5)--(3,11.5);\draw [thick,color=yellow,step=.5cm,](2.5,12.5)--(2,12.5);
\draw [thick,color=yellow,step=.5cm,](2.5,12.5)--(2.5,12);\draw [thick,color=yellow,step=.5cm,](2.5,12.5)--(3,12.5);
\draw [thick,color=yellow,step=.5cm,](2.5,12.5)--(2.5,13);\draw [thick,color=blue,step=.5cm,](3,10)--(2.5,10);
\draw [thick,color=blue,step=.5cm,](3,10)--(3,10.5);\draw [thick,color=blue,step=.5cm,](3,10)--(3.5,10);
\draw [thick,color=red,step=.5cm,](3,11)--(3,10.5);\draw [thick,color=red,step=.5cm,](3,11)--(2.5,11);
\draw [thick,color=red,step=.5cm,](3,11)--(3,11.5);\draw [thick,color=red,step=.5cm,](3,11)--(3.5,11);
\draw [thick,color=blue,step=.5cm,](3,12)--(3,11.5);\draw [thick,color=blue,step=.5cm,](3,12)--(2.5,12);
\draw [thick,color=blue,step=.5cm,](3,12)--(3,12.5);\draw [thick,color=blue,step=.5cm,](3,12)--(3.5,12);
\draw [thick,color=red,step=.5cm,](3,13)--(3.5,13);\draw [thick,color=red,step=.5cm,](3,13)--(3,12.5);
\draw [thick,color=red,step=.5cm,](3,13)--(2.5,13);
\draw [thick,color=green,step=.5cm,](3.5,10.5)--(3.5,10);
\draw [thick,color=green,step=.5cm,](3.5,10.5)--(3,10.5);
\draw [thick,color=green,step=.5cm,](3.5,10.5)--(3.5,11);
\draw [thick,color=yellow,step=.5cm,](3.5,11.5)--(3.5,11);
\draw [thick,color=yellow,step=.5cm,](3.5,11.5)--(3,11.5);
\draw [thick,color=yellow,step=.5cm,](3.5,11.5)--(3.5,12);
\draw [thick,color=green,step=.5cm,](3.5,12.5)-- (3.5,13);
\draw [thick,color=green,step=.5cm,](3.5,12.5)-- (3,12.5);
\draw [thick,color=green,step=.5cm,](3.5,12.5)--(3.5,12);
\textcolor[rgb]{0.00,0.00,1.00}{ \node (1) at (1,10)  {};}
  \node (2) at (1,10.5)  {};
  \textcolor[rgb]{1.00,0.00,0.00}{\node (3) at (1,11)  {};}
  \node (4) at (1,11.5)  {};
  \textcolor[rgb]{0.00,0.00,1.00}{\node (5) at (1,12)  {};}
  \node (6) at (1,12.5)  {};
 \textcolor[rgb]{1.00,0.00,0.00}{ \node (7) at (1,13)  {};}

\node (8) at (1.5,10)  {};
\textcolor[rgb]{0.00,1.00,0.00}{\node (9) at (1.5,10.5)  {};}
\node (10) at (1.5,11)  {};
\textcolor[rgb]{1.00,1.00,0.00}{ \node (11) at (1.5,11.5)  {};}
\node (12) at (1.5,12)  {};
\textcolor[rgb]{0.00,1.00,0.00}{\node (13) at (1.5,12.5) {};}
\node (14) at (1.5,13) {};

\textcolor[rgb]{1.00,0.00,0.00}{\node (15) at (2,10)  {};}
\node (16) at (2,10.5)  {};
\textcolor[rgb]{0.00,0.00,1.00}{\node (17) at (2,11)  {};}
\node (18) at (2,11.5)  {};
\textcolor[rgb]{1.00,0.00,0.00}{\node (19) at (2,12)  {};}
\node (20) at (2,12.5)  {};
\textcolor[rgb]{0.00,0.00,1.00}{\node (21) at (2,13)  {}; }

\node (22) at (2.5,10)  {};
\textcolor[rgb]{1.00,1.00,0.00}{\node (23) at (2.5,10.5)  {};}
\node (24) at (2.5,11)  {};
\textcolor[rgb]{0.00,1.00,0.00}{\node (25) at (2.5,11.5)  {};}
\node (26) at (2.5,12)  {};
\textcolor[rgb]{1.00,1.00,0.00}{\node (27) at (2.5,12.5)  {}; }
\node (28) at (2.5,13)  {};

\textcolor[rgb]{0.00,0.00,1.00}{\node (29) at (3,10)  {};}
\node (30) at (3,10.5)  {};
\textcolor[rgb]{1.00,0.00,0.00}{\node (31) at (3,11)  {};}
\node (32) at (3,11.5)  {};
\textcolor[rgb]{0.00,0.00,1.00}{\node (33) at (3,12)  {};}
\node (34) at (3,12.5)  {};
\textcolor[rgb]{1.00,0.00,0.00}{\node (35) at (3,13)  {};}
\node (36) at (3.5,10)  {};
\textcolor[rgb]{0.00,1.00,0.00}{\node (37) at (3.5,10.5)  {};}
\node (38) at (3.5,11)  {};
\textcolor[rgb]{1.00,1.00,0.00}{\node (39) at (3.5,11.5)  {};}
\node (40) at (3.5,12)  {};
\textcolor[rgb]{0.00,1.00,0.00}{\node (41) at (3.5,12.5)  {};}
\node (42) at (3.5,13)  {};
\draw [dashed,step=.5cm,](4.5,12)--(3.5,12);
\draw [dashed,step=.5cm,](4.5,12.5)--(3.5,12.5);
\draw [dashed,step=.5cm,](4.5,13)--(3.5,13);
\draw [dashed,step=.5cm,](4.5,11.5)--(3.5,11.5);
\draw [dashed,step=.5cm,](4.5,11)--(3.5,11);
\draw [dashed,step=.5cm,](4.5,10.5)--(3.5,10.5);
\draw [dashed,step=.5cm,](4.5,10)--(3.5,10);
\draw [dashed,step=.5cm,] (1,10)--(1,9);
\draw [dashed,step=.5cm,](1.5,10)--(1.5,9);
\draw [dashed,step=.5cm,](2,10)--(2,9);
\draw [dashed,step=.5cm,](2.5,10)--(2.5,9);
\draw [dashed,step=.5cm,](3,10)--(3,9);
\draw [dashed,step=.5cm,](3.5,10)--(3.5,9);
\node (1) at (8,10.5)  {};
\node (2) at (8,11.25)  {};
\node (3) at (8,12)  {};
\node (4) at (8,12.75)  {};
\node (5) at (8.75,10.5)  {};
\node (6) at (8.75,11.25)  {};
\node (7) at (8.75,12) {};
\node (8) at (8.75,12.75) {};
\node (9) at (9.5,10.5)  {};
\node (10) at (9.5,11.25)  {};
\node (11) at (9.5,12)  {};
\node (12) at (9.5,12.75)  {};
\node (13) at (10.25,10.5)  {};
\node (14) at (10.25,11.25)  {};
\node (15) at (10.25,12)  {};
\node (16) at (10.25,12.75)  {};
\foreach \from/\to in {1/5,5/6,6/2,6/7,7/3,7/8,8/4,7/11,6/10,9/10,10/11,11/12,12/16,11/15,10/14,9/13}
\draw (\from) -- (\to);
\node[circle,fill=gray!00] at (9.1,9.5){$\aleph$};
\node[rectangle,fill=gray!00] at (2.5,8.6){$P_{r}\Box P_{s}$};

\end{tikzpicture}
\caption{Injective edge coloring of $G = P_{r} \Box P_{s}$ which has $\chi_{i}^{'}({G})=4$ and the unicyclic graph $\aleph$, where $\chi_{i}^{'}(\aleph)=4$.}\label{grid graph}
\end{figure}
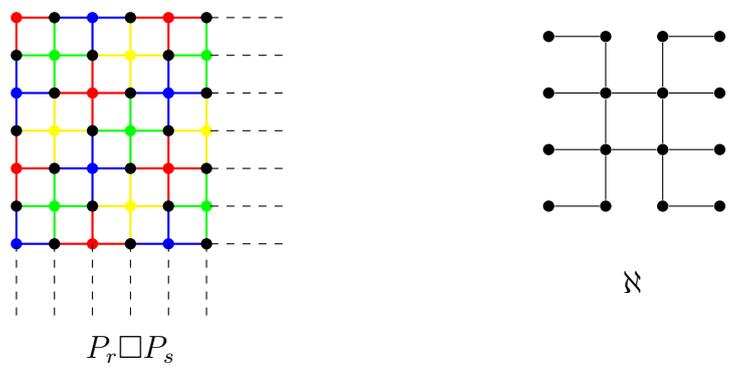
\end{center}

\begin{center}
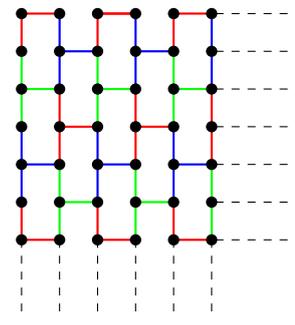
\begin{figure}[th]
\centering
\begin{tikzpicture}[every node/.append style={circle,  fill, inner sep=1.5pt, minimum size=1.5pt}]

\draw [thick,color= red, step=.5cm,] (1,10) -- (1,10.5);\draw [thick,color= blue, step=.5cm,] (1,10.5)--(1,11);
\draw [thick,color= blue, step=.5cm,] (1,11)--(1,11.5)  {};\draw [thick,color= green, step=.5cm,](1,11.5)--(1,12) ;
\draw [thick,color= green, step=.5cm,](1,12)--(1,12.5);\draw [thick,color= red, step=.5cm,](1,12.5)--(1,13);
\draw [thick,color= red, step=.5cm,](1,10)--(1.5,10);\draw [thick,color= green, step=.5cm,](1.5,10)--(1.5,10.5);
\draw [thick,color= green, step=.5cm,](1.5,10.5)--(1.5,11);\draw [thick,color= green, step=.5cm,](1.5,10.5)--(2,10.5);
\draw [thick,color= blue, step=.5cm,] (1,11)--(1.5,11);\draw [thick,color= red, step=.5cm,](1.5,11)--(1.5,11.5);
\draw [thick,color= red, step=.5cm,](1.5,11.5)--(1.5,12);\draw [thick,color= green, step=.5cm,](1.5,12)--(1,12);
\draw [thick,color= blue, step=.5cm,](1.5,12)--(1.5,12.5);\draw [thick,color= blue, step=.5cm,](1.5,12.5)--(1.5,13);
\draw [thick,color= red, step=.5cm,](1.5,13)--(1,13);\draw [thick,color= red, step=.5cm,](2,10)--(2,10.5);
\draw [thick,color= blue, step=.5cm,](2,10.5)--(2,11);\draw [thick,color= blue, step=.5cm,](2,11)--(2,11.5);
\draw [thick,color= blue, step=.5cm,](2,11)--(2.5,11);
\draw [thick,color= red, step=.5cm,](2,11.5)--(1.5,11.5);
\draw [thick,color= green, step=.5cm,](2,11.5)--(2,12);
\draw [thick,color= green, step=.5cm,](2,12)--(2,12.5);
\draw [thick,color= green, step=.5cm,](2,12)--(2.5,12);
\draw [thick,color= red, step=.5cm,](2,12.5)--(2,13);
\draw [thick,color= blue, step=.5cm,](1.5,12.5)--(2,12.5);
\draw [thick,color= red, step=.5cm,](2,13)--(2.5,13);
\draw [thick,color= red, step=.5cm,](2,10)--(2.5,10);
\draw [thick,color= green,step=.5cm,](2.5,10)--(2.5,10.5);
\node (1) at (1,10)  {};
\node (2) at (1,10.5)  {};
\node (3) at (1,11)  {};
\node (4) at (1,11.5)  {};
\node (5) at (1,12)  {};
\node (6) at (1,12.5)  {};
\node (7) at (1,13)  {};
\node (8) at (1.5,10)  {};
\node (9) at (1.5,10.5)  {};
\node (10) at (1.5,11)  {};
\node (11) at (1.5,11.5)  {};
\node (12) at (1.5,12)  {};
\node (13) at (1.5,12.5) {};
\node (14) at (1.5,13) {};
\draw [thick,color= green,step=.5cm,](2.5,10.5)--(2.5,11);\draw [thick,color= green,step=.5cm,](2.5,10.5)--(3,10.5);
\draw [thick,color= red, step=.5cm,](2.5,11)--(2.5,11.5);\draw [thick,color= red, step=.5cm,](2.5,11.5)--(3,11.5);
\draw [thick,color= red, step=.5cm,](2.5,11.5)--(2.5,12);\draw [thick,color= blue, step=.5cm,](2.5,12)--(2.5,12.5);
\draw [thick,color= blue, step=.5cm,](2.5,12.5)--(3,12.5);\draw [thick,color= blue, step=.5cm,](2.5,12.5)--(2.5,13);
\draw [thick,color= red, step=.5cm,](2.5,13)--(2,13);\draw [thick,color= red, step=.5cm,](3,10)--(3.5,10);
\draw [thick,color= red, step=.5cm,](3,10)--(3,10.5);
\draw [thick,color= blue, step=.5cm,](3,10.5)--(3,11);
\draw [thick,color= blue, step=.5cm,](3,11)--(3,11.5);
\draw [thick,color= blue, step=.5cm,](3,11)--(3.5,11);
\draw [thick,color= green,step=.5cm,](3,11.5)--(3,12);
\draw [thick,color= green,step=.5cm,](3,12)--(3.5,12);
\draw [thick,color= green,step=.5cm,](3,12)--(3,12.5);
\node (15) at (2,10)  {};
\node (16) at (2,10.5) {};
\node (17) at (2,11)  {};
\node (18) at (2,11.5)  {};
\node (19) at (2,12)  {};
\node (20) at (2,12.5)  {};
\node (21) at (2,13)  {};
\node (22) at (2.5,10)  {};
\node (23) at (2.5,10.5)  {};
\node (24) at (2.5,11)  {};
\node (25) at (2.5,11.5)  {};
\node (26) at (2.5,12)  {};
\node (27) at (2.5,12.5)  {};
\node (28) at (2.5,13)  {};
\draw [thick,color= red, step=.5cm,](3,12.5)--(3,13);
\draw [thick,color= red, step=.5cm,](3,13)--(3.5,13);
\draw [thick,color= green,step=.5cm,](3.5,10)--(3.5,10.5);
\draw [thick,color= green,step=.5cm,](3.5,10.5)--(3.5,11);
\draw [thick,color= red, step=.5cm,](3.5,11)--(3.5,11.5);
\draw [thick,color= red, step=.5cm,](3.5,11.5)--(3.5,12);
\draw [thick,color= blue, step=.5cm,](3.5,12)--(3.5,12.5);
\draw [thick,color= blue, step=.5cm,](3.5,12.5)--(3.5,13);

\node (29) at (3,10)  {};
\node (30) at (3,10.5)  {};
\node (31) at (3,11)  {};
\node (32) at (3,11.5)  {};
\node (33) at (3,12)  {};
\node (34) at (3,12.5)  {};
\node (35) at (3,13)  {};
\node (36) at (3.5,10)  {};
\node (37) at (3.5,10.5)  {};
\node (38) at (3.5,11)  {};
\node (39) at (3.5,11.5)  {};
\node (40) at (3.5,12)  {};
\node (41) at (3.5,12.5)  {};
\node (42) at (3.5,13)  {};
\draw [dashed,step=.5cm,](4.5,12)--(3.5,12);
\draw [dashed,step=.5cm,](4.5,12.5)--(3.5,12.5);
\draw [dashed,step=.5cm,](4.5,13)--(3.5,13);
\draw [dashed,step=.5cm,](4.5,11.5)--(3.5,11.5);
\draw [dashed,step=.5cm,](4.5,11)--(3.5,11);
\draw [dashed,step=.5cm,](4.5,10.5)--(3.5,10.5);
\draw [dashed,step=.5cm,](4.5,10)--(3.5,10);
\draw [dashed,step=.5cm,] (1,10)--(1,9);
\draw [dashed,step=.5cm,](1.5,10)--(1.5,9);
\draw [dashed,step=.5cm,](2,10)--(2,9);
\draw [dashed,step=.5cm,](2.5,10)--(2.5,9);
\draw [dashed,step=.5cm,](3,10)--(3,9);
\draw [dashed,step=.5cm,](3.5,10)--(3.5,9);
\end{tikzpicture}
\caption{Injective edge coloring of a honey comb graph using three colors.}\label{honeycomb graph}
\end{figure}\end{center}
Honeycomb graphs are hexagonal tessellations which appear in the literature as models of many applications.
Among several examples presented in \cite{IVS} we may emphasize the applications to cellular phone station
placement, representation of benzenoid hydrocarbons, computer graphics and image processing, etc.
\begin{prop}
If $G$ is a honeycomb graph, then $\chi_{i}^{'}(G)=3$.
\end{prop}
\begin{proof}
Since the honeycomb graph $G$ has a hexagonal tessellation, then $C_{6}$ is a subgraph $G$ and, by
Proposition~\ref{particular_cases}-\ref{case_2}, $\chi_{i}^{'} (C_{6})=3$. Therefore, considering the coloring
of the honeycomb graph $G$ presented in Figure~\ref{honeycomb graph} which (as can be easily checked)
is an injective edge coloring, it follows that  $3 = \chi_{i}^{'} (C_{6}) \le \chi_{i}^{'}(G) \le 3$.
\end{proof}

\section{Computational complexity of injective edge coloring}\label{sec_6}

The main ingredient to establish the complexity of i-egde coloring is the graph $B$ of Figure \ref{figa}, which obviously
can not be properly edge i-colored with less than 3 colors.

It can be easily verified that
\begin{lemma}\label{complexity}
In every 3 proper i-egde coloring of the graph $B$ of Figure \ref{figa} all edges incident with vertex
$b$ (and edge $e$) receive the same color.
\end{lemma}
\begin{proof}
Suppose edge $e$ and one of the edges adjacent to $e$ receive the same color, say $c_1$. Then the subgraph $B'$
of $B$, represented in Figure \ref{figb}, should be i-colored with 2 colors, as none of these edges can receive color $c_1$.
But this is clearly unfeasible. Thus, the two edges adjacent to edge $e$ should be colored with two different
colors, say $c_2\not=c_3$, both distinct from $c_1$. Clearly, the two other edges adjacent to the edge that received
color $c_2$ should also be colored with $c_2$, and the two other edges adjacent to the edge that received
color $c_3$ should also be colored with $c_3$. None of the edges incident with vertex $b$ can receive color $c_2$ nor
$c_3$, and consequently $c_1$  should be used to color all these edges.
\end{proof}

Figure \ref{figab} shows the unique (up to permutations of colors) feasible 3 i-egde coloring of graph $B$.
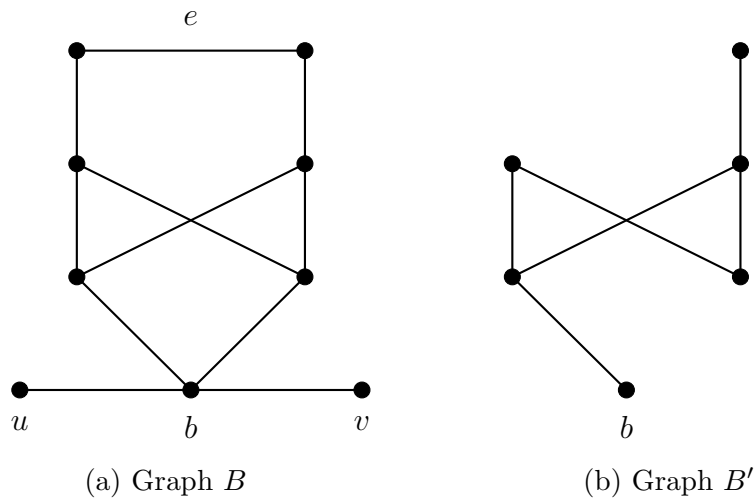
\begin{figure}[H]
   \centering	
   \begin{subfigure}[b]{0.3\textwidth}
   \begin{tikzpicture}[scale=1.5]
   \draw [thick,color= black, step=.5cm,] (-1.5,0) -- (0,0);
   \draw [thick,color= black, step=.5cm,] (0,0)--(1.5,0);
   \draw [thick,color= black, step=.5cm,] (0,0)--(-1,1);
   \draw [thick,color= black, step=.5cm,] (0,0)--(1,1) ;
   \draw [thick,color= black, step=.5cm,] (-1,1)--(-1,2);
   \draw [thick,color= black, step=.5cm,] (-1,1)--(1,2);
   \draw [thick,color= black, step=.5cm,] (1,1)--(1,2);
   \draw [thick,color= black, step=.5cm,] (1,1)--(-1,2);
   \draw [thick,color= black, step=.5cm,] (-1,2)--(-1,3);
   \draw [thick,color= black, step=.5cm,] (1,2)--(1,3);
   \draw [thick,color= black, step=.5cm,] (-1,3)--(1,3);
   \draw (-1.5,0) node[below,yshift=-2mm] {$u$} [fill] circle  [radius=0.07] ;
   \draw (0,0) node[below,yshift=-2mm] {$b$} [fill] circle  [radius=0.07];
   \draw (1.5,0) node[below,yshift=-2mm] {$v$} [fill] circle  [radius=0.07];
   \draw (0,3) node[above,yshift=+2mm] {$e$};
   \foreach \x in {1,2,3}{
       \draw (-1,\x) [fill] circle  [radius=0.07];
       \draw (1,\x) [fill] circle  [radius=0.07];
   }
   \end{tikzpicture}
   \caption{Graph $B$ }\label{figa}
   \end{subfigure}
   \hspace{2cm}
   \begin{subfigure}[b]{0.3\textwidth}
   \begin{tikzpicture}[scale=1.5]
   \draw [thick,color= black, step=.5cm,] (0,0)--(-1,1)  {};
   \draw [thick,color= black, step=.5cm,](-1,1)--(-1,2);
   \draw [thick,color= black, step=.5cm,](-1,1)--(1,2);
   \draw [thick,color= black, step=.5cm,](1,1)--(1,2);
   \draw [thick,color= black, step=.5cm,](1,1)--(-1,2);
   \draw [thick,color= black, step=.5cm,](1,2)--(1,3);
   \draw (0,0) node[below,yshift=-2mm] {$b$} [fill] circle  [radius=0.07];
   \foreach \x in {1,2}{
      \draw (-1,\x) [fill] circle  [radius=0.07];
      \draw (1,\x) [fill] circle  [radius=0.07];
   }
   \draw (1,3) [fill] circle  [radius=0.07];
   \end{tikzpicture}
   \caption{Graph $B'$}  \label{figb}

   \end{subfigure}
      \caption{Graphs $B$ and $B'$.}

\end{figure}

\begin{figure}[H]
\centering
\begin{tikzpicture}[scale=1.5]
\draw [thick,color= red, step=.5cm,] (-1.5,0) -- (0,0) node[below,midway] {$c_1$};
\draw [thick,color= red, step=.5cm,] (0,0)--(+1.5,0)  node[below,midway] {$c_1$};
\draw [thick,color= red, step=.5cm,] (0,0)--(-1,1)  node[] at (-.67,.5) {$c_1$};
\draw [thick,color= red, step=.5cm,](0,0)--(1,1)  node[] at (.7,.5){$c_1$};
\draw [thick,color= green, step=.5cm,](-1,1)--(-1,2) node[]at (-1.15,1.5)  {$c_2$};
\draw [thick,color= blue, step=.5cm,](-1,1)--(1,2) node[] at (.25, 1.8) {$c_{3}$};
\draw [thick,color= blue, step=.5cm,](1,1)--(1,2)node[] at (1.15,1.5) {$c_3$};
\draw [thick,color= green, step=.5cm,](1,1)--(-1,2)node[] at (-.47, 1.8) {$\quad c_2$};
\draw [thick,color= green, step=.5cm,](-1,2)--(-1,3) node[]  at (-1.15, 2.5){$c_2$};
\draw [thick,color= blue, step=.5cm,](1,2)--(1,3)node[] at (1.15,2.5) {$c_3$};
\draw [thick,color= red, step=.5cm,](-1,3)--(1,3)  node[below,midway] {$c_1$};
\draw (-1.5,0) node[] at (-1.5, -.25) {$u$} [fill] circle  [radius=0.07] ;
\draw (0,0) node[] at (0,-.25)  {$b$} [fill] circle  [radius=0.07];
\draw (1.5,0) node[] at (1.5,-.25) {$v$} [fill] circle  [radius=0.07];
\draw (0,3) node[above] {$e$};
\foreach \x in {1,2,3}{
\draw (-1,\x) [fill] circle  [radius=0.07];
\draw (1,\x) [fill] circle  [radius=0.07];
}
\end{tikzpicture}
\caption{The unique (up to permutations of colors) feasible 3 i-egde coloring of graph $B$. } \label{figab}
\end{figure}
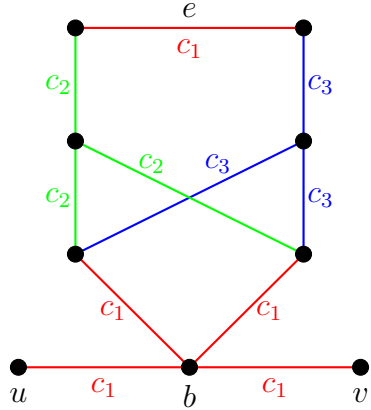

We now use Lemma \ref{complexity} and the NP-hardness of determining the edge chromatic number $\chi'(G)$ for
graphs $G$ with maximum degree $\Delta(G)=3$ (proved in \cite{Diamantis}), to establish the computational
complexity of i-egde coloring.

\begin{theorem}
It is NP-complete to recognize graphs having edge injective chromatic number equal to positive integer $k$.
\end{theorem}
\begin{proof}
Given an arbitrary graph $G$ with maximum degree $\Delta(G)=3$, we construct graph $G(B)$ replacing every
edge $uv$ of $G$ by graph $B$ of Figure \ref{figa}, such that each edge $uv$ of $G$ is now the edges $ub$
and $bv$ of $B$. Therefore, in the modified graph $G(B)$ we have a graph $B$ of Figure \ref{figa} for each
edge of $G$. Given a $k\geq 3$ edge coloring of $G$ we obtain a $k$ i-egde coloring of $G(B)$ by (i) assigning
the color used on each edge $uv$ of $G$ to edge $e$ and to every edge incident with vertex $b$
of subgraph $B$ of $G(B)$ corresponding to edge $uv$; and (ii) using any two of the remaining $k-1$ colors to color the
other edges of $G(B)$ as in Figure \ref{figab}. Clearly, the resulting $k$ i-egde coloring of $G(B)$ is feasible if and only if
the $k$ edge coloring of $G$ is feasible.

Conversely, given a feasible $k\geq 3$ i-egde coloring of $G(B)$, where each subgraph $B$ corresponding to each edge of $G$ has exactly 3
colors, we obtain a feasible $k$ edge coloring of $G$ assigning to every edge $uv$ of $G$ the color used on edge $e$ and on every edge
incident with $b$ of the subgraph $B$ of $G(B)$ corresponding to $uv$.

We thus have $\chi^{'}(G)= \chi_{i}^{'}(G(B))$. Finally, since recognizing graphs $G$ with edge chromatic number, $\chi^{'}(G)$, equal
$k$ is obviously in NP, the result follows.
\end{proof}

\section{Conclusions and open problems}
In this paper we have characterized graphs having injective chromatic index equal to one
(Proposition~\ref{prop_2}) and two (Proposition~\ref{charact_graphs}), and graphs
with injective chromatic index equal to their sizes (Proposition~\ref{complete_graphs}).
These graphs are recognized in polynomial time.
We showed that trees have injective chromatic index equal to 1, 2 or 3
(Proposition~\ref{charact_trees}), and identified the trees $T$ with
$\chi^{'}_{i}(T)=i$, for $i=1,2,3$ (Proposition~\ref{charact_trees}).

In Section~\ref{sec_3}, we have introduced the notion of $\omega^{'}EIC$-graphs (for which $\chi_i^{'}(G)=\omega^{'}(G)$) and presented a few examples of these graphs.
However, the characterizations of $\omega^{'}EIC$-graphs remains open.

Some lower and upper bounds on the injective chromatic index of a graph were obtained in Section~\ref{sec_4}.

Regarding mesh graphs, in Section~\ref{sec_5}, the injective chromatic index of the cartesian products $P_{n} \Box K_{2}$
and $ P_{r} \Box P_{s}$ as well as the honey comb graphs were determined. However, it is not known the injective chromatic
index of several other mesh graphs as it is the case of hexagonal mesh graphs (see \cite[Fig. 2]{IVS}). It is also open to compute the injective chromatic index for planar graphs.

Finally, in Section~\ref{sec_6}, we have proved that determining the injective chromatic index of graphs is NP-hard.

\section*{Acknowledgements}
This research is partially supported by the Portuguese Foundation for
Science and Technology (\textquotedblleft FCT-Funda\c c\~ao para a Ci\^encia e a Tecnologia\textquotedblright),
through the CIDMA - Center for Research and Development in Mathematics and Applications, and
within project UID/MAT/04106/2013, for the first, third and fourth authors, and
through CMA - Center of Mathematics and Applications, Project UID/MAT/00297/2013, for the second author.
The fourth author is supported by the project Cloud Thinking (CENTRO-07-ST24-FEDER-002031),
co-funded by QREN through \textquotedblleft Programa Mais Centro\textquotedblright.

\end{document}